\documentclass[12pt]{amsart}

\usepackage{hyperref}
\usepackage{amssymb}
\usepackage{amsmath}
\usepackage{amsthm}
\usepackage{enumerate}
\usepackage{graphicx}
\usepackage{mathrsfs}
\usepackage{extarrows}
\usepackage{color}
\numberwithin{equation}{section}

\usepackage{setspace}

\theoremstyle{plain}
\newtheorem{thm}{Theorem}[section]
\newtheorem{lem}[thm]{Lemma}
\newtheorem{prop}[thm]{Proposition}
\newtheorem{cor}[thm]{Corollary}

\newtheorem{conj}[thm]{Conjecture}

\theoremstyle{definition}

\theoremstyle{remark}
\newtheorem{re}{Remark}[section]

\begin{document}
	
\title[Fixed points of the uncentered Hardy-Littlewood maximal operator. ]
	{Fixed points of the uncentered Hardy-Littlewood maximal operator.}
	
\author{Wu-yi Pan}

\address{Key Laboratory of Computing and Stochastic Mathematics (Ministry of Education),
		School of Mathematics and Statistics, Hunan Normal University, Changsha, Hunan 410081, P.
		R. China}
	
\email{pwyyyds@163.com}

\date{\today}
	

	
\begin{abstract}

We give a survey, known and new results on the beingness of fixed points of the maximal operator in the more general settings of metric measure space. In particular, we prove that the fixed points of the uncentered one must be the constant function if the measure satisfies a mild continuity assumption and its support is connected.

\end{abstract}
	
\maketitle
\section{Introduction}
	Consider a metric space $(X,\rho)$, and let $\mu$ be a positive Borel measure on $X$ such that the support of $\mu$ is nonempty and the measure of each ball is finite. Throughout the short note, we will use the symbol $B$ to denote an closed ball, and $U$ for an open ball.  The \emph{Hardy–Littlewood maximal operator} centered and non-centered on $X$ and, for $f$ 
locally $\mu$-integrable, is given by
\[
M^c_{\mu} f:= \sup_{r} \frac{1}{\mu(B(x,r))}\int_{B(x,r)} fd\mu
;\quad M  _{\mu} f:= \sup_{B: B\ni x} \frac{1}{\mu(B)}\int_B fd\mu.
\]
By abuse of notation, sometimes we will let $M^c f$, $Mf$ stand for them when no confusion can arise.  In order to avoid trivialities, we set $\int_B fd\mu/\mu(B)=0$ when $\mu(B)=0$.
Let us note that maximal functions are defined everywhere in the support of $\mu$ and it does not matter whether we use open or closed balls in the definition of $M$ or $M^c$.  

A problem of fixed points of the centered maximal operators was discussed by Korry in \cite{Ko01}. His result, for $1 \leq p \leq\infty$, in the case of the Euclidean space $\mathbb{R}^d$ with Lebesgue
measure $m$,  showed that existence of a non-constant fixed point
$f\in L^p(\mathbb{R}^d)$ of $M^c_m$ will validate when $n\geq 3$ and $n/(n-2) < p \leq +\infty$ only. It is noteworthy that Fiorenza \cite{Fi88} may be the first person to try tomato—He confirmed the low dimensional case. Then the research  has attracted considerable interest (See \cite{MS06} for the rearrangement invariant space on $\mathbb{R}^d$ and \cite{Zb21} for the general $d$-dimensional Banach space). 

The Lebesgue’s differentiation theorem plays a pivotal role in their proof. It then deduces $|f|(x) \leq M^cf(x)$ almost everywhere for $f\in L^{1}_{\operatorname {loc}}(\mu)$. Yet it is not always true.  
A result of Aldaz \cite{Al18} showed that there exists a Gaussian measure on
an infinite dimensional Hilbert space for which $|f|(x) \leq M^cf(x)$ fails on a
set of positive measure. But that is unlikely to happen for finite-dimensional Banach space, thanks to  Besicovitch covering lemma.

Also let us take into account that if the geometric properties of the support of the measure are good enough, like, e.g., admitting a convex metric, the uncentred maximal operator is strictly greater than the centered one. 
More details can be seen in \cite{Pa22}. This allows us to study the fixed points question for uncentered to  centered if possible. Quite remarkably, the situation of uncentered is different. It is directly obtained from inequality of Lerner \cite{Le10} $
	\Vert M_mf\Vert_{L^p(\mathbb{R})} \geq \left(\frac{p}{p-1}\right)^{\frac{1}{p}}\Vert f\Vert_{L^p(\mathbb{R})}
$ that the fixed points of $M_m$ miss in $\mathbb{R}$. In dimension higher than one, the reverse strong $(p,p)$ inequality for the uncentered maximal operator was discovered by  Ivanisvili et al. \cite{IJN17} in 2017. It entails  the lack of nonconstant fixed
points of $M_m$, which was claimed earlier in \cite[Remark 3.13]{MS06} but they ignored the complete proof. The author and Dong \cite{PD22} showed that uniform bounds again hold for all measure satisfying  $\mu(X)=\infty$ and \begin{equation}\label{con:1.1}
	\mu(\{x\in \operatorname {supp} (\mu): r\mapsto \mu(B(x,r))\; \text{is discontinuous}\})=0.
\end{equation} on a space having the Besicovitch covering property. This result extended the previous one  to a significantly settings.

All  partial supporting evidences  may point toward a conjecture:

\begin{conj}
Let $X$ be a finite-dimensional Banach space. Then there is no non-constant fixed point
$f\in L^{1}_{\operatorname {loc}}(\mu)$ of the operator $M_\mu$ for all measure $\mu$ satisfying   \begin{equation}\label{con:1.2}
\{x\in \operatorname {supp} (\mu): r\mapsto \mu(B(x,r))\; \text{is discontinuous}\}=\varnothing.
\end{equation}
\end{conj}

Let us make a remark that the inequality we obtained in \cite{PD22}  cannot directly deduce the conjecture because the focus of our attention is from $f\in L^p(\mu)$ to $f\in L^{1}_{\operatorname {loc}}(\mu)$ and from infinite measure to no. 

Our purpose  is to show that the conjecture is true. We first state the following.

\begin{thm}\label{thm:1.2}
	Let $(X,\rho)$ be a metric space having the Heine–Borel property  and let $\mu$ be the measure on $X$ satisfying the condition of \eqref{con:1.2},  Then there is no non-constant fixed point
	$f\in L^{1}_{\operatorname {loc}}(\mu)$ of the operator $M$.
	
\end{thm}

Recall that a metric space $(X,\rho)$ is said to have the Heine–Borel property if each closed bounded set in $X$ is compact.  The most  definitive example of such space is as follows: Let $X$ be a locally compact, complete metric space such that for all $x\in X$ and $r>0$, then $\overline{U(x,r)}=B(x,r)$. See e.g. \cite{BBI01}.  This includes finite-dimensional Banach spaces and Riemannian manifolds. Also note that some infinite-dimensional Fréchet spaces have the Heine–Borel property, for instance, the space ${\displaystyle C^{\infty }(\Omega)}$ of smooth functions on an open set  $\Omega\subset\mathbb {R} ^{d}$. 

By an almost identical argument, we also prove
\begin{thm}\label{thm:1.3}
	 Assume that the connected components of  the support of $\mu$ are finite and the function $r\mapsto \mu(B(x,r))$ is continuous on the interval $(0,+\infty)$ for all fixed $x\in \operatorname {supp}(\mu)$. Then there is no non-constant fixed point
	 $f\in L^{1}_{\operatorname {loc}}(\mu)$ of the operator $M$.
\end{thm}

This useful
criteria tells us, even in some ``malformed"  metric measure spaces which do not meet the differentiation  theorem, there are still some positive results for the question. For instance, the example of Preiss \cite{Pr81} mentioned.

\begin{cor}\label{cor:1.4}
There is an infinite dimensional separable Hilbert space with a Gaussian measure $\gamma$ such that $M^cf(x) < |f|(x)$ for some $f\in L^{1}_{\operatorname {loc}}(\gamma)$ on a set of positive measure but also is no non-constant fixed point
$f\in L^{1}_{\operatorname {loc}}(\gamma)$ of the operator $M$.
\end{cor}

In our investigation here, we only consider the measure having a certain kind of continuous condition i.e.\eqref{con:1.2} because our proofs is highly depending on this condition then the interest question is  whether the restriction is not necessary. The readers can obviously find that the result is affirmative for finite discrete measures, no matter what kind of the general metric it is. The question is still open but is seen to be true.

\section{Proofs}\label{S2}

The proofs are quite unexpected, and require no priori knowledge at all. For abbreviation, we write $\langle f\rangle_A$ instead of the integral average of $f$ over a measurable set $A$: $\frac{1}{\mu(A)}\int_A|f|d\mu$.

\begin{proof}[Proofs of Theorem \ref{thm:1.2} and Theorem  \ref{thm:1.3}.]
	Fixed $\lambda>0$, consider the level set $L_\lambda:=\{x\in \operatorname {supp} (\mu): Mf(x)>\lambda\}$. The lower semi-continuous guarantees openness of $L_\lambda$.
	To prove the theorem, we first show $L_\lambda$ is also closed if nonempty. Take a point $y$ in $L^c_\lambda=\{x\in \operatorname {supp} (\mu): Mf(x)\leq\lambda\}$ (Clearly, we can suppose it is nonempty). Since the distance from $y$ to a closed set is attained, then there exists a point $x_0\in L^c_\lambda$ such that $d(y,L^c_\lambda)=d(y,x_0)$. We will denote it briefly by $d$. Thus $U(y,d)\subseteq L_\lambda=\{x\in \operatorname {supp} (\mu): f(x)>\lambda\}$ and hence $Mf(x_0)\geq\langle f\rangle_{B(y,d)}\geq \langle f\rangle_{U(y,d)}>\lambda$. This contradicts the assumption of $x_0\in L^c_\lambda$.  This means that $f$ takes a constant value on a connected component of $\operatorname {supp} (\mu)$.
 If	the connected components of $\operatorname {supp} (\mu)$ are finite  then $Mf(x)$ attains its low bound at a point $x_0\in \operatorname {supp} (\mu)$. Hence $f$ is a constant function. Otherwise we can take a ball such that $B\ni x_0$ and $
 	\mu( B\cap L_{f(x_0)})>0
$. This implies $Mf(x_0)\geq \langle f\rangle_B>f(x_0)$ which is impossible.

Now we consider the replaceable hypothesis of that  $d$ has a Heine–Borel metric. This means that each closed ball of $X$ is compact. Suppose, if possible, that $f$ is not a constant function. Thus there exists two points $x,y\in \operatorname {supp}(\mu)$ so that $f(x)>f(y)$. By  preceding part of the proof, $f(x)$ is continuous, so $Mf(x)$  also attains its low bound at a any $B\cap \operatorname {supp}(\mu)$ for which $B$ containing $x$ and $y$. It entails that $Mf(x)$ is constant on $B\cap \operatorname {supp}(\mu)$ which contradicts $f(x)>f(y).$
\end{proof}

\begin{re}
	Our  result also yields that for given $1 \leq p < \infty$ and $f\in L^p(\mu)$, 
	one can find a constant $\varepsilon>0$ such that
$
	\Vert M_\mu f\Vert_{L^p(\mu)} \geq (1+\varepsilon)\Vert f\Vert_{L^p(\mu)}
$ if we add a infinite constraint to the measure. As we have discussed earlier, Ivanisvili et al.  \cite{IJN17} demonstrated that the constant must achieve independence of $f$ when $X$ is the Euclidean space and $\mu$ is the  $d$-dimensional Lebesgue measure. Subsequently, The author and Dong take their insights to an extreme and solved the independence of the measure having \eqref{con:1.1} in \cite{PD22}. However, if one wants to determine the performance to  centered one, see \cite{IZ19,Zb21} for more details. 
\end{re}

\begin{proof}[Proof of Corollary \ref{cor:1.4}]
 Recall that the measure $\gamma$ presented in \cite{Pr81} has the following form: the restriction of the standard
 Gaussian measure from $\mathbb{R}^\mathbb{N}$ to its linear subspace
 $H:=\{x\in\mathbb{R}^\mathbb{N}:\sum_{i=1}^\infty \lambda_ix_i^2<\infty\}$ equipped with the norm $\Vert x\Vert =\left(\sum_{i=1}^\infty \lambda_ix_i^2\right)^{1/2}$ where $ \lambda_1 \geq \lambda_2 \geq \cdots > 0$ satisfy $\sum_{i=1}^{\infty}\lambda_i<\infty$ and the standard Gaussian measure on $\mathbb{R}^\mathbb{N}$ is defined as the countable
 product of the one-dimensional standard Gaussian measures. Clearly, by this special metric
 measure structure, every ball contains a small cylinder set of positive measure. Thus $\operatorname {supp} (\mu)$ is  entire space, and hence the corollary follows by applying Theorem \ref{thm:1.3}.
\end{proof}

Finally, we conclude this note with a applications of our main results, which is an attempt to find out whether two maximal operators coincide at a function.

\begin{prop}\label{pr:2.1} Let $(X,\rho)$ be a metric space having the Heine–Borel property and let $\mu$ be a measure on $X$ satisfying the condition of \eqref{con:1.2}.
Assume the support of $\mu$ with no isolated points. Then there is no function $f\in L^{1}_{\operatorname {loc}}(\mu)$ such that $M^cf=Mf$ and every local minimum point of $Mf$ is  strict local minimum.
\end{prop}

We also left a conjecture for the reader to verify.
\begin{conj}
	Let $X$ be a finite-dimensional Banach space; let $m$ be the Lebesgue measure on $X$ and $f\in L^{1}_{\operatorname {loc}}$. Then $M^cf=Mf$ if and only if $f\in L^\infty$ and $\lim\limits_{r\to \infty}\langle f\rangle_{B(x,r)}=\Vert f\Vert_\infty$.
\end{conj}
\begin{re}
The necessity is indubitable since $Mf(y)\geq \lim\limits_{r\to \infty}\langle f\rangle_{B(x,r)}$ for all $x,y$. Also note that $M^cf=Mf=\Vert f\Vert_\infty$ in the case.
\end{re}

To prove Proposition \ref{pr:2.1}, we need one simple fact.
\begin{lem}\label{thm:2.3}
	If there is a measure such that $M^cf(x)> Mf(y)$ for some $x,y \in \operatorname {supp} (\mu)$, then $Mf(x)$ is attained or $Mf(x)=\lim\limits_{r\to 0}\langle f\rangle_{B(x,r)}$. 
\end{lem}
\begin{proof}
	To this end, we prove 
	$
	M^cf(x)=\sup\limits_{r\in[0,d(x,y)]}\langle f\rangle_{B(x,r)},
	$ then the lemma follows by passing to a subsequence.  We prove it by contradiction. Suppose, if possible, that the extreme radii diverge to infinity or is attained for some $r>d(x,y)$. If $M^cf(x)<+\infty$, so for any $\varepsilon$, we can select a $r_\varepsilon>d(x,y)$ such that $\langle f\rangle_{B(x,r_\varepsilon)}\geq M^cf(x)-\varepsilon$. Now  $B(x,r_\varepsilon)\ni y$ implies $Mf(y)\geq M^cf(x)-\varepsilon$ which obviously contradicts $M^cf(x)> Mf(y)$.
	
	Next, we have to consider the case of $M^cf(x)=+\infty$.  The same of
	reduction to absurdity shows a contradicting to $Mf(y)<+\infty$.
\end{proof}

\begin{proof}[Proof of Proposition \ref{pr:2.1}]
	Observe first that the measure assumed is non-atomic. For fixed $x\in \operatorname {supp}(\mu)$, consider the biggest radius $r_x$ so that all point $y\in B(x,r_x)$ has  $Mf(y)\geq Mf(x)$.  Thus $x$ is a local minimum point  if $r_x >0.$ Since the strict minimum point of a function is a countable set, combine the assumption then $\{x\in \operatorname {supp}(\mu):r_x=0\}$ is a set of full measure. Now applying Lemma \ref{thm:2.3}, we get $Mf(x)=\lim\limits_{r\to 0}\langle f\rangle_{B(x,r)}$ holds almost everywhere. The Lebesgue’s differentiation theorem demonstrates $Mf(x)=f(x)$ a.e.. This means we can reassign its function value on a set of measure zero to ensure $Mf(x)=f(x)$ everywhere. Now Theorem \ref{thm:1.2} implies that $f$ is a constant function and hence so does for $Mf$. This contradicts the assumption of $Mf$, so we completed the proof.
\end{proof}


\begin{thebibliography}{999}
\bibitem[Al18]{Al18} J. M. Aldaz,  \textit{On the pointwise domination of a function by its maximal function.} Arch. Math. (Basel) 111 (2018), no. 3, 299–311.	

\bibitem[BBI01]{BBI01}D. Burago, Y. Burago, S. Ivanov, 
\textit{A course in metric geometry.}
Graduate Studies in Mathematics, 33. American Mathematical Society, Providence, RI, 2001. xiv+415 pp.


\bibitem[Fi88]{Fi88} A. Fiorenza,  \textit{A note on the spherical maximal function.} Rend. Accad. Sci. Fis. Mat. Napoli (4). 54 (1987), 77–83.

\bibitem[IJN17]{IJN17} P. Ivanisvili,  B. Jaye, F. Nazarov,  \textit{Lower bounds for uncentered maximal functions in any dimension.} Int. Math. Res. Not. IMRN (2017), no.8, 2464–2479.

\bibitem[IZ19]{IZ19}P. Ivanisvili, S. Zbarsky Centered Hardy-Littlewood maximal operator on the real line: lower bounds. C. R. Math. Acad. Sci. Paris 357 (2019), no. 4, 339–344.
	
\bibitem[Ko01]{Ko01}S. Korry, \textit{Fixed points of the Hardy–Littlewood maximal operator.} Collect. Math. 52(2001), no. 3,  289–294.


\bibitem[Le10]{Le10}  A. K. Lerner, \textit{Some remarks on the Fefferman-Stein inequality.} J. Anal. Math. 112(2010), 329–349.

\bibitem[MS06]{MS06}J. Martín, J. Soria,  \textit{Characterization of rearrangement invariant spaces with fixed points for the Hardy-Littlewood maximal operator.} Ann. Acad. Sci. Fenn. Math. 31 (2006), no. 1, 39–46.

\bibitem[Pa22]{Pa22}W. Pan, \textit{A remark on the Hardy-Littlewood maximal operators.} Available at the Mathematics ArXir. 
	
\bibitem[PD22]{PD22}W. Pan, X. Dong, \textit{Lower bounds for uncentered maximal functions on metric measure space.} Available at the Mathematics ArXir. 

\bibitem[Pr81]{Pr81}Preiss, D, \textit{Gaussian measures and the density theorem.} Comment. Math. Univ.
Carolin. 22, 181–193 (1981)

\bibitem[Zb21]{Zb21}S. Zbarsky, \textit{Lower bounds and fixed points for the centered Hardy-Littlewood maximal operator.} J. Geom. Anal. 31 (2021), no. 1, 817–824.



\end{thebibliography}
\end{document}